\documentclass{amsart}
\usepackage{amsmath,amsthm,amssymb,IMjournal}

\usepackage[dvips]{graphicx}

\usepackage{url}

\newtheorem{remark}{Remark}[section]

\newtheorem{thm}{Theorem}[section]

\theoremstyle{definition}

\theoremstyle{remark}

\numberwithin{equation}{section}
\numberwithin{lem}{section}
\numberwithin{thm}{section}

\numberwithin{equation}{section}

\begin{document}

\title{Explicit Estimation of Error Constants Appearing in Non-conforming Linear Triangular Finite Element}

\author{\|Xuefeng |LIU|, Niigata, Japan, ~~~
        \|Fumio |KIKUCHI|, Tokyo, Japan}


\dedicatory{Cordially dedicated to ...}

\abstract 

The non-conforming linear ($P_1$) triangular FEM can be viewed as a kind of the discontinuous Galerkin method, and is attractive in both theoretical and practical senses. 
Since various error constants must be quantitatively evaluated for its accurate a priori and a posteriori error
estimates, we derive their theoretical upper bounds and some computational results. In particular,
the Babu$\check{s}$ka-Aziz maximum angle condition is required just as in the case of the conforming $P_1$ triangle. 
Some applications and numerical results are also illustrated to see the validity and effectiveness of our analysis.

\endabstract



\thanks
This research is supported by Grants-in-Aid for Scientific Research (JSPS KAKENHI) (C) 26800090, (B) 16H03950 (the first author)
and (C)(2) 16540096, (C)(2)19540115 (the second author) from Japan Society for the Promotion of Science (JSPS).
\endthanks

\section{Introduction}
\label{sec1}
\footnote{This paper is a revision of the original one  \cite{liu-kikuchi-2007} in proceedings of APCOM'07 conference in conjunction with EPMESC XI, the only digital version of which is however not easy to find.}
As a well-known alternative to the conforming linear ($P_1$) triangular finite element for approximation of 
the first-order Sobolev space ($H^1$), the non-conforming $P_1$ element is considered a classical discontinuous Galerkin finite element \cite{arnold-1985-2} and has various interesting properties from both theoretical and practical 
standpoints \cite{ciarlet-2002,temam-1973}. In particular, its a priori error analysis was performed in fairly early stage of mathematical analysis of FEM (Finite Element Method), and recently a posteriori error analysis is rapidly developing as well. 
For accurate error estimation of such an FEM, various error constants must be evaluated quantitatively \cite{ainsworth-2000,babuska-2001,brenner-2002,ciarlet-2002}.

Based on our preceding works on the constant ($P_0$) and the conforming $P_1$ triangles \cite{kikuchi-liu-2006,kikuchi-liu-2007}, we here give some results for error constants required for analysis of the non-conforming $P_1$ triangle. 
More specifically, we first summarize a priori error estimation of the present non-conforming FEM, where several error constants 
appear. In this process, we use the lowest-order Raviart-Thomas triangular $H$(div) element to deal with the inter-element discontinuity of the approximate functions \cite{brezzi-1991,kikuchi-2007}. Then we introduce some constants related to a reference triangle, some of which are popular in the $P_0$ and the conforming $P_1$ cases. We give some theoretical results for the upper bounds of such constants. Finally, we illustrate some numerical results to support the validity of such upper bounds.
Our results can be effectively used in the quantitative a priori and a posteriori error estimates for the non-conforming $P_1$ triangular FEM.

\section{A PRIORI ERROR ESTIMATION}
\label{sec2}

We here summarize a priori error estimation of the non-conforming $P_1$ triangular FEM. Let $\Omega$ be a bounded convex polygonal domain in $R^2$ with boundary $\partial \Omega$, and let us consider a weak 
formulation of the Dirichlet boundary value problem for the Poisson equation:
Given $f\in L_2(\Omega)$, find $u\in H_0^1(\Omega)$ s.t. 
\begin{equation}
\label{poisson-pro}
(\nabla u, \nabla v) = (f,v) ;~~  \forall v \in H_0^1(\Omega).
\end{equation}

Here, $L_2(\Omega)$ and $H_0^1(\Omega)$ are the usual Hilbertian Sobolev spaces associated to $\Omega$,
$\nabla$ is the gradient operator, and $(\cdot,\cdot)$ stands for the inner products for both $L_2(\Omega)$ and
 $L_2(\Omega)^2$. It is well known that the solution exists uniquely in $H^1_0(\Omega)$ and also belongs to
 $H^2(\Omega)$ for the considered $\Omega$.

Let us consider a regular family of triangulations $\{\mathcal{T}^h\}_{h>0}$ of $\Omega$, to which we associate the non-conforming $P_1$ finite element spaces $\{V^h\}_{h>0}$. Each $V^h$ is constructed over a certain $\mathcal{T}^h$,
and the functions in $V^h$ are linear in each $K\in \mathcal{T}^h$ with continuity only at midpoints of edges, and also vanish at the midpoints on the $\partial \Omega$ to approximate the homogeneous Dirichlet condition \cite{ciarlet-2002,temam-1973}.
Then the finite element solution $u_h \in V^h$ is determined by, for a given $f\in L_2(\Omega)$,
\begin{equation}
\label{nonc-weak}
(\nabla_h u_h, \nabla_h v_h) = (f,v_h) ; ~~ \forall v_h \in V^h
\end{equation}
where $\nabla_h$ is the ``non-conforming" or discrete gradient defined as $L_2(\Omega)^2$-valued operator by the element-wise 
relations $(\nabla_h v)|_K = \nabla(v|_K)$ for $\forall v \in V^h+H^1(\Omega)$ and $\forall K \in \mathcal{T}^h$.
Equation (\ref{nonc-weak}) is formally of the same form as in the conforming case, so that, for error analysis, it is natural to 
consider an appropriate interpolation operator $\Pi_h$, e.g., the Crouzeix-Raviart interpolation, from $H_0^1(\Omega)$ (or its intersection with some other spaces)  to 
$V^h$. However, the situation is not so simple. That is, using the Green formula, we have
\begin{equation}
\label{discontinuity}
(\nabla_h u_h, \nabla_h v_h)  = ( \nabla_h u, \nabla_h v_h ) - \sum_{K\in \mathcal{T}^h}\int_{\partial K} v_h \frac{\partial u}{\partial n}|_{\partial  K} d\gamma; ~ \forall v_h \in V^h, 
\end{equation}
where $\frac{\partial u}{\partial n}|_K$ denotes the trace of the derivative of $u$ in the outward normal direction of $\partial K$,
and $d\gamma$ does the infinitesimal  element of $\partial K$. Conventional efforts of error analysis have been focused
on the estimation of the second term in the right-hand side of  (\ref{discontinuity}), which is absent in the conforming case.
To cope with such difficulty, we introduce the lowest-order Raviart-Thomas triangular $H$(div) finite element space $W^h$ associated
to each $\mathcal{T}^h$ \cite{brezzi-1991,kikuchi-2007}. 
Then, noticing that the normal component of $\forall q_h \in W^h$ is constant and continuous along each inter-element edge, 
we can derive 
\begin{equation}
\label{eq:nonc-hdiv}
(q_h,\nabla _h v_h) + (\text{div } q_h,v_h) =0\:.
\end{equation}
From (\ref{nonc-weak}) and (\ref{eq:nonc-hdiv}), 
\begin{equation}
\label{gradient-error}
(\nabla_h u_h, \nabla_h v_h) = (q_h, \nabla_h v_h)+(\text{div } q_h + f,v_h) \:.  
\end{equation}
By puting $-(\nabla u, \nabla_h v_h)$ on both hand sides of (\ref{gradient-error}), we have, for any $ q_h \in W^h $, $v_h \in V^h$, 
\begin{equation}
\label{gradient-error}
(\nabla_h u_h -\nabla u, \nabla_h v_h) = (q_h- \nabla u, \nabla_h v_h)+(\text{div } q_h + f,v_h) \:. 
\end{equation}

Then by Lemma 6 of \cite{kikuchi-1975}, a refinement of Strang's second lemma\cite{ciarlet-2002}, we have
\footnote{ The proof restricted to (\ref{nonc-weak}) is simple. Let $P_{h}$ be the projection that project $V$ to $V^{h}$, with respect to $(\nabla_{h}\cdot, \nabla_{h}\cdot)$. 
Then $|| \nabla_h u_h - \nabla u  ||^2 = \|\nabla_h P_{h} u - \nabla u \|^{2}+\|\nabla_h (u_h - P_{h} u) \|^{2}$. 
Noticing that $\|\nabla_h (u_h - P_{h} u) \|^{2}=(\nabla_h (u_h - P_{h} u), \nabla_h u_h - \nabla u)$ and applying (\ref{gradient-error}), we can easily get (\ref{priori-est}).
}
\begin{equation}
\label{priori-est}
|| \nabla u -\nabla_hu_h ||^2 = \inf_{v_h\in V^h} || \nabla u - \nabla_h v_h||^2  + 
\left[   \sup_{w_h\in V^h \setminus \{0 \} } \frac{(q_h-\nabla u, \nabla_h w_h) + ( \text{div } q_h + f, w_h)}{|| \nabla_h w_h ||}\right]^2
\:,
 \end{equation}
where $||\cdot||$ stands for the norms of both $L_2(\Omega) $ and  $L_2(\Omega)^2 $.
Using the Fortin operator 
$\Pi_h^F: H(\text{div } ;\Omega) \cap H^{\frac{1}{2}+\delta}(\Omega)^2 \rightarrow W^h(\delta >0)$ (cf. \cite{brezzi-1991}) and the orthogonal projection one $Q_h:L_2(\Omega) \rightarrow X^h:=$ space of step functions over $\mathcal{T}^h$,
we obtain a priori error estimate:
\begin{equation}
\label{priori-est-fortin}
|| \nabla u -\nabla_h u_h  ||^2 \le \inf_{v_h \in V^h} || \nabla u -\nabla_h v_h ||^2 
+ \left[ 
|| \nabla u -\Pi_h^F \nabla u || + 
 \sup_{w_h\in V^h \setminus \{0 \} } \frac{(f-Q_h f, w_h-Q_h w_h)}{|| \nabla_h w_h ||}
\right]^2
\end{equation}
where $q_h$ in (\ref{priori-est}) is taken as $\Pi^F_h \nabla u$.

We can obtain a more concrete error estimate in terms of the mesh parameter
 $h_\ast >0$ (see definition of $h_\ast$ in (\ref{def:mesh_h}); $h$ will be used in a different meaning later)  by deriving estimates such as for $\forall v \in H_0^1(\Omega)\cap H^2(\Omega)$ and 
$\forall g \in H^1(\Omega)+ V^h$,
\begin{equation}
\label{interpolation-est}
\begin{array}{rr}
|| v-\Pi_h v || \le \gamma_0 h_\ast^2 |v|_2, & 
|| \nabla v - \nabla_h \Pi_h v || \le \gamma_1 h_\ast  |v|_2 \\
|| \nabla v- \Pi_h^F \nabla  v || \le \gamma_2 h_\ast |v|_2, & 
|| g - Q_h g || \le \gamma_3 h_\ast  || \nabla_h g||
\end{array}
\end{equation}
where $|\cdot|_k$ denotes the standard semi-norm of $H^k(\Omega)$ ($k \in \mathbf{N}$), and $\gamma_i$'s are
positive error constants dependent only on $\{\mathbf{T}^h\}_{h>0}$.

Then we obtain, for the solution $u\in H^1_0(\Omega)\cap H^2(\Omega)$,
$$
|| \nabla u - \nabla_h u_h || \le 
\left\{
\begin{array}{ll}
h_\ast \{ \gamma_1^2 |u|_2^2 + (\gamma_2 |u|_2 + \gamma_3 ||f||)^2  \}^{1/2}&
\text{ for } f\in L_2(\Omega),\\
h_\ast \{ \gamma_1^2 |u|_2^2 + (\gamma_2 |u|_2 + \gamma_3^2 h_\ast|f|_1)^2  \}^{1/2} &
\text{ for } f\in H^1(\Omega),
\end{array}
\right.
$$
where the term $|u|_2$ can be bounded as $|u|_2 \le ||f||$ for present $\Omega$.\\

We can also use Nitsche's trick to evaluate a priori $L_2$ error of $u_h$\cite{ciarlet-2002,knabner-2003}. 
That is, let us define $\psi \in H^1_0(\Omega)(\cap H^2(\Omega))$ for $e^h:=u-u_h$ by 
$$
(\nabla \psi,\nabla v) = (e^h,v); \forall v \in H_0^1(\Omega),
$$
Then for  $\forall v_h \in V^h$ and $\forall q_h, \tilde{q}_h \in W^h$, by noticing 

$$
\left\{
\begin{array}{l}
 || e^h||^2 = (e^h, e^h) = (\text{div } \tilde{q}_h  + e^h, e^h) + (\tilde{q}_h, \nabla_h e^h  )\\
 (-\nabla_h v_h, \nabla_h e^h) 
+ (\nabla_h v_h, \nabla u) + (-v_h, f) =0 \\
 (- \nabla \psi, \nabla u - q_h ) + (\psi, \text{div } q_h + f  ) =0 \\
 (\nabla_h v_h, -q_h) + (v_h, \text{div }q_h )=0
\end{array}
\right.
$$
we have,
\begin{multline*}
|| e^h||^2 = (\tilde{q}_h -\nabla_h v_h, \nabla_h e^h ) +
(\nabla_h v_h - \nabla \psi, \nabla u - q_h) +\\
(\psi -v_h, \text{div } q_h +f ) + ( \text{div } \tilde{q}_h + e^h,e^h).
\end{multline*}
Substituting $v_h=\Pi_h \psi, q_h = \Pi_h^F \nabla u$ and $\tilde{q}_h=\Pi_h^F \nabla \psi$ above, we find
\begin{multline*}
|| e^h||^2 = (\Pi_h^F \nabla \psi - \nabla\psi +\nabla\psi  - \nabla_h \Pi_h \psi, \nabla_h e^h )
+ (\nabla \Pi_h \psi -\nabla \psi,\nabla u-\Pi_h^F \nabla u)
+\\
(\psi - \Pi_h\psi, f-Q_hf) + (e^h-Q_h e^h,e^h-Q_h e^h)\:,
\end{multline*}
since $\text{div }q_h$=div $\Pi_h^F \nabla u=  -Q_h f$ and div $\tilde{q}_h$ = div $\Pi_h^F\nabla \psi = -Q_h e^h$. Then we have, by  (\ref{interpolation-est}) as 
well as the relations $|u|_2 \le  ||f||$ and $|\psi|_2\le ||e^h||$,
$$
||e^h||^2 \le \left[(\gamma_1+\gamma_2)h_\ast || \nabla_h e^h || + (\gamma_0 + \gamma_1 \gamma_2) h_\ast^2 ||f|| \right] 
||e^h|| + 
\gamma_3^2 h_\ast^2 || \nabla e^h ||^2 \:,
$$
where the term $\gamma_0 h_\ast^2 ||f|| \cdot ||e^h||$ can be replaced with $\gamma_0\gamma_3 h_\ast^3 |f|_1 ||e^h||$ if $f\in H^1(\Omega)$.
This may be considered a quadratic inequality for $e^h$, and solving it gives an expected order estimate 
$||u-u_h||=||e^h||=O(h_\ast^2)$:
$$
||e^h|| \le \frac{h_\ast}{2} (A_1 + \sqrt{A_1^2+4A_2}); \:
A_1:=(\gamma_1+\gamma_2) ||\nabla_h e^h|| 
+ (\gamma_0 + \gamma_1 \gamma_2) h_\ast ||f||,
A_2:=\gamma_3^2 h_\ast || \nabla e^h ||^2 
$$

\vskip .5cm
{\bf RELATION TO RAVIART-THOMAS MIXED FEM}
\vskip .5cm

We have already introduced the Raviart-Thomas space $W^h$ for auxiliary purposes.
But it is well known that the present non-conforming FEM is closely related to 
the Raviart-Thomas mixed FEM \cite{arnold-1985,marini-1985}. Here
we will summarize the implementation of such a mixed FEM by slightly modifying the original nonconforming $P_1$ scheme described by  (\ref{nonc-weak}). The original idea in \cite{arnold-1985,marini-1985} is based on the enrichment by the conforming 
cubic bubble functions with the $L_2$ projection into $W^h$, but we here
adopt non-conforming quadratic bubble ones to make the modification procedure a little simpler.\footnote{In 2015, Hu \& Ma show the same result about the relation between the enriched FEM and Raviart-Thomas FEM, along with the extension to general dimensional space \cite{Hu2015}.}

Firstly, we replace $f$ in  (\ref{nonc-weak}) by $Q_h f$. Then $u_h$ is modified to $u_h^\ast \in V^h$ defined by 
\begin{equation}
\label{modified-weak-form}
(\nabla_h u_h^\ast, \nabla_h v_h) = (Q_h f, v_h) ; \: \forall  v_h \in V^h.
\end{equation}
Secondly, we introduce the space $V^h_B$ of non-conforming quadratic bubble 
functions by defining its basis function $\varphi_K$ associated to each $K\in \mathcal{T}^h$: $\varphi_K$ vanished outside $K$ and its value at $x\in K$ is given by 
\begin{equation}
\label{modify-fun}
\varphi_K(x) = \frac{1}{2}|x-x^G|^2 - \frac{1}{12} \sum_{i=1}^{3} |x^{(i)}-x^G|^2,
\end{equation}
where $|\cdot|$ is the Euclidean norm of $\mathbf{R}^2$, $x^G$ the barycenter of $K$, and $x^{(i)}$ ($i=1,2,3$) the $i$-th vertex of $K$. It is easy to see that the line integration of $\varphi_K$ for each $e$ of $K$ vanishes:
\begin{equation}
\label{integration-constraint}
\int_e \varphi_K ~ \mbox{d}\gamma =0\:.
\end{equation}
Now the enriched non-conforming finite element space $\tilde{V}^h $ is defined by
the following linear sum:
\begin{equation}
\tilde{V}^h = V^h \oplus V_B^h \:.
\end{equation}
By  (\ref{integration-constraint}) and the Green formula, we find the following orthogonality relation for $(\nabla_h \cdot, \nabla_h \cdot)$:
\begin{equation}
\label{orthogonality}
(\nabla_h v_h,\nabla_h \beta_h) = 0; \: \forall v_h \in V^h, \forall \beta_h \in V_B^h\:.
\end{equation}
Then the modified finite element solution $\tilde{u}_h\in \tilde{V}^h$ is defined by
\begin{equation}
\label{new-fem-weak}
(\nabla_h \tilde{u}_h, \nabla_h \tilde{v}_h) = (Q_h f, \tilde{v}_h); \: \forall \tilde{v}_h \in \tilde{V}^h\:.
\end{equation}

Thanks to  (\ref{orthogonality}), the present $\tilde{u}_h$ can be obtained 
as the sum:
\begin{equation}
\label{solution-descomp}
\tilde{u}_h = u_h^\ast +\alpha_h\:,
\end{equation}
where $u_h^\ast \in V^h$ is the solution of (\ref{modified-weak-form}), and 
$\alpha_h \in V_B^h$ is determined by 
\begin{equation}
\label{complement-equ}
(\nabla_h \alpha_h, \nabla_h \beta_h) = (Q_h f, \beta_h); \: \forall \beta_h \in V_B^h,
\end{equation}
i.e., completely independently of $u_h^\ast$. Moreover, $\alpha_h$ can be decided by element-by-elment comupations. More specifically, denoting $\alpha_h|_K$ as $\alpha_K \varphi_K|K$,   (\ref{complement-equ}) leads to 
\begin{equation}
\alpha_K (\nabla \varphi_K, \nabla \varphi_K)_K  = (Q_h f, \varphi_K)_K ;\: \forall K \in \mathcal{T}^h,
\end{equation}
where $(\cdot, \cdot)$ denotes the inner products of both $L_2(K)$ and $L_2(K)^2$.

Define $\{p_h, \overline{u}_h\} \in L_2(\Omega)^2 \times X^h$ by
\begin{equation}
\label{mixed-fem-solution}
p_h = \nabla_h \tilde{u}_h, \: \overline{u}_h = Q_h \tilde{u}_h\:.
\end{equation}
By appying the Green formula to  (\ref{new-fem-weak}), we can show that $p_h\in W^h$, and also that the present pair $\{p_h,\overline{u}_h \}$ satisfies the 
determination equations of the lowest-order Raviart-Thomas mixed FEM:
\begin{equation}
\label{mixed-fem}
\left\{
\begin{array}{ll}
(p_h,q_h)  + (\overline{u}_h,\text{div } q_h) =0 ; & \forall q_h \in W^h,\\
(\text{div } p_h, \overline{v}_h) = - (Q_h f, \overline{v}_h); & \forall \overline{v}_h \in X^h.
\end{array}
\right.
\end{equation}
By the uniqueness of the solutions, $\{p_h,\overline{u}_h \}$ is nothing but the unique solution of  (\ref{mixed-fem}).

In conclusion, denoting the constant value of $Q_hf|K$ by $\overline{f}_K \left(=\int_Kf\:\mbox{d}x/\text{meas}(K) \right)$, we have for $\forall K\in \mathcal{T}^h$ and
$\forall x \in K$ that
\begin{equation*}
\left\{
\begin{array}{l}
 \alpha_K = -\displaystyle{\frac{1}{2}} \overline{f}_K,\: \\
 \displaystyle{
\tilde{u}_h (x) = u_h^\ast(x) + \alpha_K \psi_K(x)=
u_h^\ast(x) -\frac{1}{4}\overline{f}_K(|x-x^G|^2 -\frac{1}{6}\sum_{i=1}^3|x^{(i)}-x^G|^2),}
\end{array}
\right.
\end{equation*}
and
\begin{equation}
\label{mixed-fem-solution-detail}
\left\{
\begin{array}{l}
\displaystyle{
p_h(x) = \nabla_h u_h^\ast (x) -\frac{1}{2}\overline{f}_K(x-x^G),
}
\\
\displaystyle{
\overline{u}_h(x) = u_h^\ast(x^G) - \frac{1}{16} \overline{f}_K(|x^G|^2-\frac{1}{3}\sum_{i=1}^3|x^{(i)}|^2),
}
\end{array}
\right.
\end{equation}
which coincide with those in \cite{marini-1985} and are easy to compute by post-processing.

\vskip 0.5cm

{\bf  A Posteriori error estimation}

The consideration in the preceding section suggests the a posteriori error estimation based 
on  the hypercircle method \cite{destuynder-1999,kikuchi-2007}.

Taking  notice the fact that $p_h \in W^h$ obtained in the preceding section belongs to 
$H(\text{div };\Omega)$ with $\text{div } p_h=-Q_h f$, we find that, for 
 $\forall v \in H_0^1(\Omega)$,
\begin{equation}
\label{hypercircle}
|| \nabla v -p_h||^2 =|| \nabla (v-u^h) ||^2 + || \nabla u^h -p_h||^2, \quad
||\nabla u^h - \frac{1}{2}(\nabla v + p_h) || = \frac{1}{2}|| \nabla v -p_h ||\:,
\end{equation}
where $u^h\in H_0^1(\Omega)$ is the solution of (\ref{poisson-pro}) with $f$
replaced by $Q_h f$:
\begin{equation}
(\nabla u^h,\nabla v)= (Q_hf,v); \: \forall v\in H^1_0(\Omega)\:.
\end{equation}
 (\ref{hypercircle}) implies that the three points $\nabla u^h$,
$\nabla v$ and $p_h$ in $L_2(\Omega)^2$ make a hypercircle, the first having a right inscribed angle. Noting that $(f-Q_f,v)=(f-Q_hf, v-Q_hv)$ for 
$\forall v\in H_0^1(\Omega)\subset L_2(\Omega)$, we have by 
(\ref{priori-est-fortin}) that
\begin{equation}
|u-u^h|_1=|| \nabla (u-h^h)||
 \le \gamma_3 h_\ast || f-Q_hf|| \quad (\le \gamma_3^2 h_\ast^2 |f|_1 \text{ if } f \in H^1(\Omega) )\:.
\end{equation}
Taking an approxpriate $v\in H_0^1(\Omega)$, we obtain a posteriori error estimates related $p_h=\nabla _h \tilde{u}_h$:
\begin{equation}
|| \nabla u -p_h  || \le 
|| \nabla (u -u^h)  || + || \nabla u^h -p_h  || \le 
|| \nabla (u-u^h)   || + || \nabla v -p_h ||
\end{equation}
\begin{equation}
|| \nabla u - \frac{1}{2}(\nabla v + p_h) || \le
|| \nabla(u-u^h) || +  \frac{1}{2} || \nabla v -p_h || \:.
\end{equation}
A typical example of $v$ is the conforming $P_1$ finite element solution
 $u_h^C\in V_C^h$, where $V_c^h$ is the conforming $P_1$ space over 
$T^h$. Another example is a function $v_C^h\in V_C^h$ obtained by appropriate post-processing of $u_h$ or $u_h^\ast$, such 
as nodal averaging or smoothing. A cheap method of constructing a nice $v_C^h$ may be also an interesting subject. 
Again, we need the constant $\gamma_3$ to evaluate the term $||\nabla(u-u^h) ||$ above.
If we use $\nabla_hu_h$ based on the original $u_h \in V^h$ in (\ref{nonc-weak}), instead of the modified one $\tilde{u}_h\in V^h$, we must 
evaluate some additional terms. Fortunately, such evaluation can be done explicitly by using $\gamma_3$ and some positive constants related to 
$\{ \psi_K \}_{K\in \mathcal{T}^h}$.\\

{\bf Error Constants}
To analyze the error constants in (\ref{priori-est-fortin}), let us consider their 
element-wise counterparts. Let $h$, $\alpha$ and $\theta$ be positive
constants such that 
\begin{equation}
\label{def-range}
h>0,\quad 0<\alpha\le 1,\quad  (\frac{\pi}{3} \le ) \cos ^{-1} \frac{\alpha}{2} \le \theta < \pi \:.
\end{equation}
Then we define the triangle $T_{\alpha,\theta,h}$ by $\triangle OAB$ with
three vertices $O(0,0),A(h,0)$ and $B(\alpha h \cos \theta, \alpha h \sin \theta)$.
From (\ref{def-range}), $AB$ is shown to be the edge of maximum length, i.e.,
$\overline{AB} \ge h \ge \alpha h$, so that $h=\overline{OA}$ here
denotes the medium edge length, unlike the usual usage as the largest one \cite{ciarlet-2002}. A point on the closure $\overline{T}_{\alpha,\theta,h}$
is denoted by $x=\{x_1,x_2\}$, and the three edges $e_i$'s ($i=1,2,3$) are defined by $\{e_1,e_2,e_3\}=\{OA,OB,AB\}$. 

By an appropriate congruent
transformation in $\mathbf{R}^2$, we can configure any triangle as 
$T_{\alpha,\theta,h}$. As the usage in \cite{babuska-1976}, we will use abbreviated notations $T_{\alpha,\theta}=T_{\alpha,\theta,1}$, 
$T_{\alpha}=T_{\alpha,\frac{\pi}{2}}$ and $T=T_1$ 
(Fig.\ref{triangle-configure}). We will also use the notations $||\cdot||_{T_{\alpha,\theta,h}}$ and $|\cdot|_{k,T_{\alpha,\theta,h}}$
 as the norms of $L_2(T_{\alpha,\theta,h})$ and semi-norms of $H^k(T_{\alpha,\theta,h})$, where the subscript $T_{\alpha,\theta,h}$ will be usually omitted.
 
\begin{figure}[th]
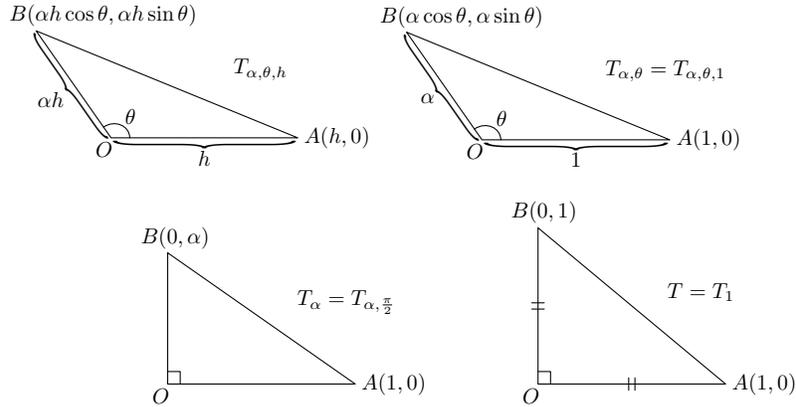

\label{triangle-configure}
\begin{picture}(200,170)(50,20)
\put(0,110){\includegraphics[width=1.9in]{tri1.ps}} 
\put(140,110){\includegraphics[width=1.9in]{tri2.ps}}
\put(50,20){\includegraphics[width=1.5in]{tri3.ps}}
\put(190,20){\includegraphics[width=1.5in]{tri4.ps}}
\end{picture}
\caption{Notations for triangles:  $T_{\alpha,\theta}=T_{\alpha,\theta,1}$, 
$T_{\alpha}=T_{\alpha,\frac{\pi}{2}}$, $T=T_1$ }
\end{figure}

Let us define the following closed linear spaces fro functions over $T_{\alpha,\theta,h}$:
\begin{eqnarray}
V_{\alpha,\theta,h}^0 & = & \{ v\in H^1(T_{\alpha,\theta,h}) \:|\: \int_{T_{\alpha,\theta,h}}v(x)\:\mbox{d}x=0\}
\end{eqnarray}
\begin{eqnarray}
V_{\alpha,\theta,h}^i & = & \{ v\in H^1(T_{\alpha,\theta,h})  \:|\: \int_{e_i}v(s)\:\mbox{d}s=0 \} \quad (i=1,2,3)\\
V_{\alpha,\theta,h}^{\{1,2\}} & = & \{ v\in H^1(T_{\alpha,\theta,h})  \:|\: \int_{e_1}v(s)\:\mbox{d}s=\int _{e_2}v(s)\:\mbox{d}s=0\}\\
V_{\alpha,\theta,h}^{\{1,2,3\}} & = & \{ v\in H^1(T_{\alpha,\theta,h})  \:|\: \int_{e_i}v(s)\:\mbox{d}s=0\quad (i=1,2,3)\}\\
\label{space-v4}
V_{\alpha,\theta,h}^4 & = & \{ v\in H^2(T_{\alpha,\theta,h}) |\int_{e_i}v(s)\:\mbox{d}s=0\quad (i=1,2,3)\}
\end{eqnarray}
We will again use abbreviations like $V_{\alpha,\theta}^0=V_{\alpha,\theta,1}^0$,
$V_{\alpha}^0=V_{\alpha,\frac{\pi}{2}}^0$, $V^0=V_{1}^0$, etc.\\

Let us consider the $P_0$ interpolation operator $\Pi_{\alpha,\theta,h}^0$ and non-conforming $P_1$ one $\Pi_{\alpha,\theta,h}^{1,N}$ for functions on $T_{\alpha,\theta,h}$ \cite{brenner-2002,ciarlet-2002}: $\Pi_{\alpha,\theta,h}^0 v$ for  $\forall v \in H^1(T_{\alpha,\theta,h})$ is a constant function such that
\begin{equation}
(\Pi_{\alpha,\theta,h}^0 v)(x) = \frac{ \int_{T_{\alpha,\theta,h}} v(y) \mbox{d}y }{ |T_{\alpha,\theta,h}|},\quad \forall x\in T_{\alpha,\theta,h},
\end{equation}
while $\Pi_{\alpha,\theta,h}^{1,N} v$ for $\forall v \in H^1(T_{\alpha,\theta,h})$ is a linear function such that 
\begin{equation}
\int_{e_i} (\Pi_{\alpha,\theta,h}^{1,N} v)(s)ds=\int_{e_i}v(s)ds \text{ for } i=1,2,3.
\end{equation}

To analyze these interpolation operators, let us estimate the positive constants defined by 
\begin{eqnarray}
C_J(\alpha,\theta,h)&\!\!=&\!\!\!\!\sup_{v\in V_{\alpha,\theta,h}^J\setminus \{0\}} \frac{ ||v||}{|v|_1} \quad (J=0,1,2,3,\{1,2\},\{1,2,3\}),\\
C_4(\alpha,\theta,h)&\!\!=&\!\!\!\!\sup_{v\in V_{\alpha,\theta,h}^4\setminus \{0\}}\frac{|v|_1}{|v|_2}, \:\:
C_5(\alpha,\theta,h)=\!\!\!\!\!\sup_{v\in V_{\alpha,\theta,h}^4\setminus \{0\}}\frac{||v||}{|v|_2}\:.
\end{eqnarray}

We will again use abbreviated notations $C_{J}(\alpha,\theta)=C_J(\alpha,\theta,1)$,
$C_J(\alpha)=C_J(\alpha,\pi/2)$, $C_J=C{(1)}$ and also $C_{J,\alpha,\theta}:=C_J(\alpha,\theta)$
for every possible subscript $J$.

By a simple scale change, we find that $C_J(\alpha,\theta,h)=hC(\alpha,\theta)(J\ne5)$  and 
$C_5(\alpha,\theta,h)=h^2 C_5(\alpha,\theta)$. Now, by noticing $v-\Pi_{\alpha,\theta,h}^0 v \in V_{\alpha,\theta,h}^0$
 for $v\in H^1(T_{\alpha,\theta,h})$ and 
$v-\Pi_{\alpha,\theta,h}^{1,N} v \in V_{\alpha,\theta,h}^4$ for $v\in H^2(T_{\alpha,\theta,h})$, 
we can easily have the popular interpolation error estimates on $T_{\alpha,\theta,h}$\cite{brenner-2002,ciarlet-2002}.
\begin{eqnarray}
\label{nonc-interp-est-c0}
|| v-\Pi_{\alpha,\theta,h}^0 v|| \le C_0(\alpha,\theta) h |v|_1; \quad \forall v \in H^1(T_{\alpha,\theta,h})\\
\label{nonc-interp-est-c4}
| v-\Pi_{\alpha,\theta,h}^{1,N} v|_1 \le C_4(\alpha,\theta) h |v|_2;  \quad \forall v \in H^2(T_{\alpha,\theta,h})\\
\label{nonc-interp-est-c5}
|| v-\Pi_{\alpha,\theta,h}^{1,N} v|| \le C_5(\alpha,\theta) h^2 |v|_2;  \quad  \forall v \in H^2(T_{\alpha,\theta,h})
\end{eqnarray}
We can show that the following relations hold for the constants $C_{J,\alpha,\theta}\left(:=C_J(\alpha,\theta)\right)$:
\begin{equation}
\label{nonc-constant-est}
C_{4,\alpha,\theta} \le C_{0,\alpha,\theta},\quad C_{5,\alpha,\theta}\le C_{0,\alpha,\theta}C_{\{1,2,3\},\alpha,\theta} \le C_{0,\alpha,\theta} C_{\{1,2\},\alpha,\theta}\:.
\end{equation}
An estimation rougher than the latter of  (\ref{nonc-constant-est}) is 
$C_{5,\alpha,\theta} \le C_{0,\alpha,\theta} \min_{i=1,2,3}C_{i,\alpha,\theta}$.
To show former of (\ref{nonc-constant-est}), we first derive $\int_{T_{\alpha,\theta}} \partial v/\partial x_i \mbox{d}x =0$ 
for $\forall v \in V_{\alpha,\theta}^4 (i=1,2)$ by considering the definition in  (\ref{space-v4}) and 
applying the Gauss formula.
Then we can easily obtain the desired result by noticing the definition of $C_0(\alpha,\theta)$. 
To  derive the latter of (\ref{nonc-constant-est}), we should evaluate $||v||/|v|_1$ and $|v|_1/|v|_2$ for
$\forall v \in V_{\alpha,\theta}^4 \:(i=1,2)$. The former quotient can be evaluated by using $C_{\{1,2,3\}}(\alpha,\theta)$, while the latter can be done by $C_4(\alpha,\theta)$ together with 
former of (\ref{nonc-constant-est}).
Clearly, $C_{\{1,2,3\}}(\alpha,\theta)\le C_{\{1,2\}}(\alpha,\theta)$, then we have the latter of
 (\ref{nonc-constant-est}).

Thus we can give quantitative interpolation estimates from (\ref{nonc-interp-est-c0}) throught 
(\ref{nonc-interp-est-c5}), if we succeed in evaluating or bounding the constants $C_J(\alpha,\theta)$'s
explicitly for all possible $J$. Among them, $C_0(\alpha,\theta)$ and $C_{\{1,2\}}(\alpha,\theta)$
are important as may be seen from (\ref{nonc-constant-est}). Notice that each of such constants
can be characterized by minimization of a kind of Rayleigh quotient \cite{babuska-1976,nakao-2001,nakao-2001-2}. Then it is equivalent to finding the minimum eigenvalue of a certain eigenvalue problem expressed by a weak formulation for a partial differential equation with some auxiliary conditions.

Moreover, we already derived some results for $C_{i}(\alpha,\theta)$ for $i=0,1,2$ (\cite{kikuchi-liu-2006,kikuchi-liu-2007}) 
\footnote{K. Kobayashi also develops upper bounds for the error constants; see, e.g., \cite{Kobayashi2011,Kobayashi2015}.}. In particular, $C_0=1/\pi$, and $C_1(=C_2)$
is equal to the maximum positive solution of the equation $1/\mu + \tan (1/\mu)=0$ for $\mu$.
The constants $C_J(\alpha,\theta)$'s for $J=0,1,2,3,4,5,\{1,2\},\{1,2,3\}$ are bounded uniformly for $\{\alpha,\theta\}$. 
More specifically, their explicit upper bounds are given in terms of $\alpha,\theta$ and 
their values at $\{\alpha,\theta\}=\{1,\pi/2\}$. Furthermore,$C_J(\alpha)$'s except for $J=4$
are monotonically increasing in $\alpha$.
Asymptotic behaviors of the constants $C_J(\alpha)$'s for $\alpha\downarrow 0$ 
can be also analyzed in \cite{kikuchi-liu-2007}. As a result, the interpolation by the non-conforming
$P_1$ triangle is robust to the distortion of $T_{\alpha,\theta}$. This fact does not necessarily imply
the robustness of the final error estimates for $u-u_h$, since analysis of the Fortin interpolation has not been performed yet.

\begin{remark}
Instead of $\Pi_{\alpha,\theta,h}^{1,N}$, it is also possible to consider an interpolation operator using the function values at midpoints of edges. Such an operator is definable for continuous functions over 
$\overline{T}_{\alpha,\theta,h}$, but not so for functions in $H^1(T_{\alpha,\theta,h})$.
Moreover, its analysis would be different from the for $\Pi_{\alpha,\theta,h}^{1,N}$.
\end{remark}

{\bf Determination of $C_{\{1,2\}}$}
From the preceding observations, we can give explicit upper bounds of various interpolation constants associated to the non-conforming $P_1$ triangle, provided that the value of $C_{\{1,2\}}$ is determined. This becomes indeed possible by adopting essentially the same idea and techniques to determine
$C_0$ and $C_1(=C_2)$:
\begin{thm}
$C_{\{1,2\}}=C_{\{1,2\}}(1,\pi/2,1)$ is equal to the maximum positive solution of the transcendental equation for $\mu$:
\begin{equation}
\label{determ-equ}
\frac{1}{2\mu} + \tan\frac{1}{2\mu} =0
\end{equation}
The above implies that $C_{\{1,2\}}=\frac{1}{2}C_1(=\frac{1}{2}C_2)$, and hence is bounded as, with numerical verification,
\begin{equation}
0.24641 < C_{\{1,2\}} < 0.24647\:.
\end{equation} 

\end{thm}
\begin{remark} Thus $1/4$ is a simple but nice upper bound.	Numerically, we have
$C_{\{1,2\}}=0.2464562258\cdots$. 

\end{remark}
\begin{proof} By the use of the technique for determination of $C_0$ and $C_1=C_2$
in \cite{kikuchi-liu-2006,kikuchi-2007}, we obtain the following equation for $\mu$:
\begin{equation}
1+\frac{1}{2\mu}\sin\frac{1}{\mu} - \cos \frac{1}{\mu}=0\:,
\end{equation}
whose maximum positive solution is the desired $C_{\{1,2\}}$. By the double-angle
formulas, the above is transformed into 
\begin{equation}
(2\sin\frac{1}{2\mu} + \frac{1}{\mu} \cos\frac{1}{2\mu}) \sin\frac{1}{2\mu}=0\:.
\end{equation}
It is now easy to derive (\ref{determ-equ}), and also to draw other conclusions by using the resutls in \cite{kikuchi-liu-2006,kikuchi-2007}.
\end{proof}

{\bf Analysis of Fortin's interpolation }
This section is devoted to the analysis of Fortin's interpolation operator $\Pi_{\alpha,\theta}^F$ (\cite{brezzi-1991}) for 
each $T_{\alpha,\theta}$ . Given 
$q\in H(\text{div }; T_{\alpha,\theta}) \cap H^{\frac{1}{2}+\delta}(T_{\alpha,\theta})^2 (\delta>0) $,
the Fortin interpolation ${q_h}  = \{ \alpha_1 + \alpha_3 x_1, \alpha_2 + \alpha_3 x_2 \}$ ($\alpha_i$ being constants) satisfies,
$$
\int_{e_i} ({q_h} - q )\cdot \vec{n} \: \text{d} s =0, \quad i=1,2,3\:.
$$

To consider the error estimation for Fortin's interpolation, we quote a result about the error estimation for 
the Lagrange interpolation function. Define constant $C_F$ by
$$
C_F := \sup_{ q\in W(T_{\alpha,\theta}) } \frac{\|q \|}{|q|_1}\:.
$$
Here $W(T_{\alpha,\theta})$ is defined by
$$
W(T_{\alpha,\theta}):=\{ q \in H(T_{\alpha,\theta})^2 \:| \: \int_{e_i}q\cdot \vec{\tau} \mbox{d}s =0, i=1,2,3.\},
$$
where $\vec{\tau}$ denotes the unit tangent vector along edges.
Such a constant has been used to bound the Lagrange interpolation error constant (Theorem 2 of \cite{liu-kikuchi-2010}), 
which has  an explicit upper bound $C_F \le C_6(\alpha,\theta)$ as follows.
\begin{equation}
\label{eq:fortin-est-constant}
C_6(\alpha,\theta) := \frac{\left\{ c_1^2 + c_2^2+ 2c_1c_2 \cos^2 \theta +
(c_1 + c_2) \sqrt{c_1^2+c_2^2+2c_1c_2 \cos 2 \theta} \right\}^{1/2}}{\sqrt{2}\sin \theta} 
\end{equation}
where $c_i$ presents $C_i(\alpha,\theta) (i=1,2)$ for the purpose of abbreviation.

The following theorem gives the error constant for the Fortin interpolation, 
where the technique in the proof is following the one used in 
Theorem 5.1 of \cite{Carstensen2012} \footnote{The result below is an improvement of the error estimation of \cite{liu-kikuchi-2007}, which involves another constant $C_{7}$ along with the term $\|\mbox{div }q\|$, which however can be removed. }.

\begin{thm}
It holds for $q=\{q_1,q_2\} \in \left(H^1(T_{\alpha,\theta})\right)^2$ that 
\begin{gather}
\label{fortin-final-est}
\| q-\Pi_{\alpha,\theta}^{F}q \| \le C_6(\alpha,\theta) |q|_1 \: .
\end{gather}
\end{thm}

\begin{proof}
Let $\hat{w}$ be the rotation of $w:=q-\Pi^F_{\alpha,\theta}q$ by $\pi/2$, 
then it is easy to verify that $\int_{e_i} \hat{w} \cdot \vec{\tau}~ \mbox{d}s=0$, $i=1,2,3$.
Hence, 
$$
(\|w\| = )~~\|\hat{w}\| \le C_6(\alpha,\theta) |\hat{w}|_1 ~~(= C_6(\alpha,\theta) |{w}|_1)\:.
$$
Rewrite the vector $w$ by $w=(w_1, w_2)$ and decompose $|w|_1^2$ by
$$
|w|_1^2 = \| w_{1,x}- \frac{\text{div }w}{2} \| ^2 + \| w_{1,y} \|^2 + \| w_{2,x} \|^2  + \| w_{2,y} - \frac{\text{div }w}{2} \|^2  + \|\text{div }w\|^2/2\:.
$$
Also, noticing that for $q_h=(q_{h_1},q_{h_2}):=\Pi^F_{\alpha,\theta}q$,
$$ q_{h_1,x}-\frac{\text{div } q_h}{2} =  q_{h_2,y}- \frac{\text{div }{q_h}}{2} = q_{h_1,y} = q_{h_2,x}=0$$
and the orthogonal decomposition of $\text{div }w$,
$$
\| \text{div }w \|^2 + \| \text{div } q_h \|^2 = \|\text{div }q\|^2\:,
$$
we have $|w|_1^2 \le |q|_1^2$, which leads to the conclusion. 
\end{proof}

\begin{remark} 
Because of the factor $\sin \theta$ in (\ref{eq:fortin-est-constant}), the maximum angle condition applies to 
estimate (\ref{fortin-final-est}) \cite{acosta-duran-2000,babuska-1976,kikuchi-2007}.
On the other hand, the estimates for $\Pi^0_{\alpha,\theta,h}$ and $\Pi^{1,N}_{\alpha,\theta,h}$ are free
from such conditions as may be seen from (\ref{nonc-constant-est}) and the comments there.\\
\end{remark}

{\bf GLOBAL INTERPOLATION OPERATORS}\\

So far, we have introduced and analyzed local interpolation operators 
$\Pi^0_{\alpha,\theta,h}, \Pi^{1,N}_{\alpha,\theta,h}$ and $\Pi^F_{\alpha,\theta,h}$.
For each $K\in \mathcal{T}^h$, we can find an appropriate $T_{\alpha,\theta,h}$ congruent to $K$
under a mapping $\Phi_K:K \to T_{\alpha,\theta,h}$. Then it is natural to define the $P_1$
non-conforming interpolation operator $\Pi_h:H_0^1(\Omega)\to V^h $ by 
$\Pi_h u |_{K} = \Pi_{\alpha,\theta,h}^{1,N} ( v|_{K} \circ \Phi_K^{-1}) \circ \Phi_K$
for $\forall v \in H_0^1(\Omega)$ and $\forall K \in \mathcal{T}^h$. 
Similarly, the orthogonal projection operator $Q_h:L_2(\Omega)\to X^h$ is related to $\Pi^0_{\alpha,\theta,h}$, while the global Fortin operator $\Pi_h^F$
is defined through $\Pi^F_{\alpha,\theta,h}$, $\Phi_K$ and the Piola transformation for 2D contravariant vector fields \cite{arnold-1985}.

For each $K\in \mathcal{T}^h$, define $\{\alpha_K,\theta_K,h_K \}$ as $\{\alpha,\theta,h\}$ of the associated $T_{\alpha,\theta,h}$.
Then, our analysis shows that the estimates in (\ref{interpolation-est}) can be concretely given by, for 
$\forall v \in H_0^1(\Omega)\cap H^2(\Omega)$ and $\forall g\in H^1(\Omega) + V^h$,
\begin{equation*}
\label{interpolation-est}
\begin{array}{ll}
\| v-\Pi_h v \| \le C_5^h h_{\ast}^2 |v|_2 \le C_0^h C_{\{1,2\}}^h h_{\ast}^2 |v|_2, & 
\| \nabla v - \nabla \Pi_h v \| \le C_4^h h_\ast  |v|_2 \le C_0^h h_\ast  |v|_2 \\
\| \nabla v- \Pi_h^F \nabla v \| \le C_6^h  h_\ast |v|_2, & 
\| g - Q_h g \| \le C_0^h h_\ast  \| \nabla_h g \|
\end{array}
\end{equation*}
where 
\begin{equation}
\label{def:mesh_h}
h_\ast = \max_{K\in T^h} h_K, \quad C_J^h:=\max_{K\in \mathcal{T}^h} C_J(\alpha_K,\theta_K) \quad (J=0,4,5,6,7,\{1,2\}).
\end{equation}

\begin{remark} Relations such as (\ref{solution-descomp}), (\ref{mixed-fem-solution}) and (\ref{mixed-fem-solution-detail})
may suggest the possibility of finding interpolations for $\nabla u$ in $W^h$ than the one by the Fortin operator, which 
are free from the maximum angle condition \cite{babuska-1976}. However, $\nabla_h(\Pi_hu+\alpha_h)$, for example, 
is not shown to belong to $W^h$, because we cannot prove the inter-element continuity of normal components unlike 
$\nabla_h \hat{u}_h$. Our numerical results show that the maximum angle condition is probably essential for the non-conforming 
$P_1$ triangle. See also \cite{acosta-duran-2000} for related topics.\\
\end{remark}

{\bf NUMERICAL RESULTS}\\

Firstly, we performed numerical computations to see the actual dependence of various constants on $\alpha$ and $\theta$
by adopting the conforming $P_1$ element and a kind of discrete Kirchhoff plate bending element \cite{kikuchi-1995}, the latter of which is used to deal with directly the 4-th
order partial differential eigenvalue problems related to $C_4(\alpha,\theta)$ and $C_5(\alpha,\theta)$. That is,
we obtained some numerical results for $C_4(\alpha)$ and $C_5(\alpha)$ ($\theta=\pi/2$) together with their upper bounds. 
We used the uniform triangulation of the entire domain $T_\alpha$ : $T_\alpha$ is subdivided into small triangles, all being congruent to 
$T_{\alpha,\pi/2,h}$ with e.g. $h=1/20$.

The left-hand side of Fig.\ref{fig-c4-c5} illustrates the graphs of approximate $C_4(\alpha)$ and $C_0(\alpha)$
versus $\alpha \in ]0,1]$, while the right-hand side does similar graphs for $C_5(\alpha)$ and $C_0(\alpha)C_{\{1,2\}}(\alpha)$.
In both cases, the theoretical upper bounds based on () give fairly good approximations to the considered constants $C_4(\alpha)$
and $C_5(\alpha)$. Asymptotic behaviors of the constants for $\alpha \downarrow 0+$ observed in the figures can be analyzed as in 
\cite{kikuchi-liu-2007}.

\begin{figure}[h]
\label{fig-c4-c5}
\includegraphics[width=4.5in]{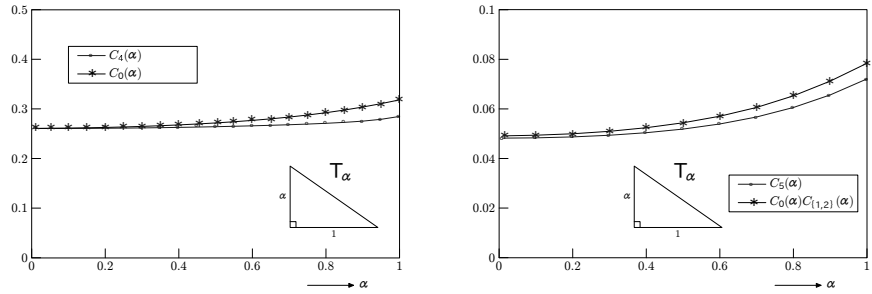}
\caption{ Numerical results for $C_{4}(\alpha)$ \& $C_{0}(\alpha)$ (left), and for $C_{5}(\alpha)$ \& $C_{0}(\alpha)C_{\{1,2\}}(\alpha)$ (right); $0<\alpha \le 1$ }
\end{figure}

We also tested numerically the validity of our a priori error estimate for $\| \nabla u - \nabla_h u_h \|$. That is,we
choose $\Omega$ as the unit square $\{x=\{x_1,x_2\}; 0<x_1,x_2<1\}$ and $f$ as $f(x_1,x_2) = \sin \pi x_1 \sin \pi x_2 $,
and consider the $N \times N$ Friedrichs-Keller type uniform triangulations $(N\in \mathcal{N})$. In such situation, 
$u(x_1,x_2)=\frac{1}{2\pi^2} \sin \pi x_1 \sin \pi x_2$, and all the triangles are congruent to a right isosceles triangle
$T_{1,\pi/2,1/N}$, i.e., $h_\ast=h=1/N$. Moreover, we can use the following values or upper bounds for necessary constants:
\begin{gather*}
C_0^h = C_0 = 1/\pi,\quad C_{\{1,2\}}^h = C_{\{1,2\}} < 1/4, \quad 
 C_6^h=C_1=C_2< 1/2. 
\end{gather*}
Moreover, under current boundary condition and domain shape, we have $|u|_2=\|\Delta u\|=\|f\|$; see, e.g., 
Theorem 4.3.1.4 of \cite{grisvard2011}. Then, since $f\in H^1(\Omega)$, the a priori error estimation is given as,
$$
\|\nabla u -\nabla_h u_h\| \le h_\ast \left( \frac{1}{\pi^2} \|f\|^2 + (\frac{1}{2}\|f\|+ \frac{h_\ast}{\pi^2} \|\nabla f\| )^2 \right)^{1/2}
$$

\begin{figure}[hbtp]
\label{numerical-error-est}
\includegraphics[width=2.5in]{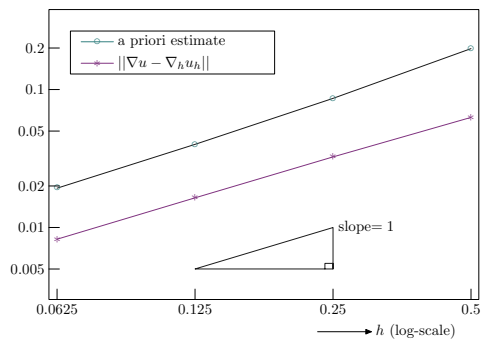}
\caption{ $\| \nabla u - \nabla_h u_h\|$ and its a priori estimates v.s $h$ }
\end{figure}

Figure \ref{numerical-error-est} illustrates the comparison of the actual $\|\nabla u - \nabla _h u_h\|$ and 
its a priori estimate based on our analysis. The difference is still large, but anyway our analysis appears to give
correct upper bounds and order of errors. Probably, a posteriori estimation mentioned previously would give more realistic results. \footnote{Another kind a priori error estimation is given in 
\cite{Carstensen2012}, which gives larger (worse) estimation compared to our proposed estimation, if the two estimations are applied to the example used in \cite{Carstensen2012}.}\\

\noindent
{\bf CONCLUDING REMARKS}

\noindent
We have obtained some theoretical and numerical results for several error constants associated to the non-conforming $P_1$ triangle.
These results are hoped to be effectively used in quantitative error estimates, which are necessary for adaptive mesh refinements 
\cite{bangerth2013adaptive} and numerical verifications. Especially for numerical verification of partial differential equations by Nakao's method \cite{nakao-2001}, accurate bounding of various error constants is essential. Moreover, we are planning to extend our analysis 
to its 3D counterpart, i.e., the non-conforming $P_1$ tetrahedron with face DOF's.



\bibliographystyle{plain}
\bibliography{liu}

\vskip 0.2cm

{\small

\noindent
{\em Authors' addresses}:

\noindent
{\em Xuefeng LIU}~~ Graduate School of Science and Technology, Niigata University, 8050 Ikarashi 2-no-cho, Nishi-ku, Niigata City, Niigata, 950-2181, Japan; 
e-mail: \texttt{xfliu@\allowbreak math.sc.niigata-u.ac.jp}\\
\noindent
{\em Fumio KIKUCHI}~~ Graduate School of Mathematical Sciences, University of Tokyo, 3-8-1, Komaba, Meguro, Tokyo, 153-8914, Japan; e-mail: \texttt{kikuchi@\allowbreak ms.u-tokyo.ac.jp}
}

\end{document}